\documentclass[a4paper,11pt]{amsart}

\usepackage{amssymb,amsmath,url,float}   
\usepackage[numbers, sort&compress]{natbib}
\usepackage[inner=3cm,outer=3cm,top=3cm,bottom=3cm,includehead,includefoot]{geometry}
\numberwithin{equation}{section}

\newtheorem{thm}{Theorem}[section]

\newtheorem{cor} [thm]{Corollary}
\newtheorem{lemma} [thm]{Lemma}

\theoremstyle{definition}

\renewcommand\leq{\leqslant} 
\renewcommand\geq{\geqslant}

\DeclareMathOperator{\End}{End}

\DeclareMathOperator{\GL}{GL}
\DeclareMathOperator{\M}{M}

\DeclareMathOperator{\J}{J}
\DeclareMathOperator{\HS}{HS}
\DeclareMathOperator{\Suz}{Suz}
\DeclareMathOperator{\McL}{McL}
\DeclareMathOperator{\Ru}{Ru}
\DeclareMathOperator{\He}{He}
\DeclareMathOperator{\Ly}{Ly}
\DeclareMathOperator{\ON}{O'N}
\DeclareMathOperator{\Co}{Co}
\DeclareMathOperator{\Fi}{Fi}
\DeclareMathOperator{\HN}{HN}
\DeclareMathOperator{\Th}{Th}
\DeclareMathOperator{\Baby}{\mathbb{B}}
\DeclareMathOperator{\Mon}{\mathbb{M}}

\DeclareMathOperator{\Aut}{Aut}
\DeclareMathOperator{\Out}{Out}

\title{Regular orbits of sporadic simple groups}

\author{Joanna B. Fawcett}
\address[Fawcett]{Department of Mathematics,  Imperial College London, South Kensington Campus, 
London, SW7 2AZ, United Kingdom}
\email{j.fawcett@imperial.ac.uk}

\author{J\"urgen M\"uller}
\address[M\"uller]{ 
Fakult\"at f\"ur Mathematik und Naturwissenschaften,
Bergische Universit\"at, Wuppertal, Germany}
\email{juergen.mueller@math.uni-wuppertal.de}

\author{E.\ A.\ O'Brien}
\address[O'Brien]{Department of Mathematics, University of Auckland, Auckland, 
New Zealand}
\email{e.obrien@auckland.ac.nz}

\author{Robert A.\ Wilson}
\address[Wilson]{School of Mathematical Sciences, Queen Mary, University of London, 
Mile End Road, London, E1 4NS, United Kingdom}
\email{r.a.wilson@qmul.ac.uk}


\thanks{Fawcett was supported by the Australian Research Council  Discovery Project 
grant DP130100106, the London Mathematical Society, and  the European Union's Horizon 2020 research and innovation programme under the Marie Sk\l{}odowska-Curie grant agreement No.\ 746889”. O'Brien was supported 
by a University of Auckland Hood Fellowship. Both they and Wilson were 
supported by the Marsden Fund of New Zealand via grant UOA 1323. We thank the referees for their  helpful and insightful comments.}
\begin{document}

\begin{abstract}
 Given a finite group $G$ and a faithful irreducible $FG$-module $V$ where $F$ has prime order, does $G$ have a regular orbit on  $V$? This problem is equivalent to determining which primitive permutation groups of affine type have a base of size 2. 
Let $G$ be a covering group of an almost simple group whose socle $T$ is   
sporadic, and let $V$ be a faithful irreducible $FG$-module 
where $F$ has prime order dividing $|G|$.
We classify the pairs $(G,V)$ for which $G$ has no regular orbit on $V$,  
and determine the minimal base size of $G$ in its action on $V$.
To obtain this classification, 
for each non-trivial $g\in G/Z(G)$, we compute  the minimal number of 
$T$-conjugates of $g$  generating  $\langle T,g\rangle$.  \end{abstract}

\maketitle

\section{Introduction}
\label{s:intro}

A \textit{base} $B$ for a group $X$ acting faithfully on a finite set $\Omega$ is a  subset of $\Omega$ with the property that only the identity of $X$ fixes every element of $B$. The \textit{base size} of $X$, denoted by $b(X)$,  is the minimal cardinality of a base for $X$. Recently, much work has been done 
to classify the finite primitive permutation groups of almost simple, diagonal and twisted wreath type  with base size 2 (see
\cite{BurGurSax2011,BurGurSaxS,BurObrWil2010,Faw2013,FawPhD}). 
For groups of affine type, this problem 
is equivalent to the regular orbit problem for fields with prime order.

Given a finite group $G$, a field $F$ and a faithful $FG$-module $V$, we say that $G$ has a \textit{regular orbit}  on $V$ if there exists $v\in V$ such that only the identity of $G$ fixes $v$; in other words, $\{v\}$ is a base for $G$. Hall, Liebeck and Seitz \cite[Theorem 6]{HalLieSei1992} proved   that  if $G$ is a finite quasisimple group with no regular orbit on a faithful irreducible $FG$-module $V$ where $F$ is a field of characteristic $p$, then either $G$ is of Lie type in characteristic $p$, or $G= A_n$ where $p\leq n$ and $V$ is the fully deleted permutation module, or $(G,V)$ is one of finitely many exceptional pairs. While these exceptional pairs are unknown in general, they have been determined when $F$ is the field $\mathbb{F}_p$ of  order $p$ and either $p\nmid |G|$ 
(see \cite{Goo2000,KohPah2001}), or $p\mid |G|$ and $G/Z(G)=A_n$
(see~\cite{FawObrSax2016}).  In this paper, we consider the case where  $G/Z(G)$ is a  sporadic simple group whose order is divisible by $p$. We also consider the covering groups of the automorphism groups of the sporadic groups, and for those groups $G$ with no regular orbit on $V$, we determine the base size of $G$ in its action on $V$.

\begin{thm}
\label{regular}
Let $G$ be a covering group  of an almost simple group whose socle is sporadic.  
Let $V$ be a faithful irreducible $\mathbb{F}_pG$-module where 
$p$ is a prime dividing $|G|$.  If $G$ has no regular orbit on $V$, 
then $(G,p,\dim_{\mathbb{F}_p}(V),b(G))$ is listed in Table $\ref{tab:totalex}$.
\end{thm}

If there are exactly $m$ faithful irreducible $\mathbb{F}_pG$-modules with dimension $d$ on which $G$ has no regular orbit and $m>1$, then we write  $d^{(m)}$  in
Table~\ref{tab:totalex}. Except for the case $(G,p,\dim_{\mathbb{F}_p}(V))=(\M_{11},3,10)$, this is sufficient to identify the non-regular modules. 
However, there are three faithful irreducible $\mathbb{F}_3\M_{11}$-modules of dimension $10$, only one of which has base size $2$: the $\mathbb{F}_3\M_{11}$-module  with the property that an involution of $\M_{11}$, viewed as an element of $\GL_{10}(3)$, has trace $-1\in\mathbb{F}_3$.

There are $\mathbb{F}_pG$-modules in Table~\ref{tab:totalex}
(indicated by ${}^\sharp$) that are 
not absolutely irreducible; these all split into 
absolutely irreducible $\mathbb{F}_{p^2}G$-modules.
In particular, when $G=3.\J_3$, there is a unique faithful irreducible 
but not absolutely irreducible $\mathbb{F}_2G$-module with dimension $18$,
corresponding to the two absolutely irreducible $\mathbb{F}_4G$-modules 
with dimension $9$ that cannot be realised over $\mathbb{F}_2$; 
this module should not be confused with the four faithful absolutely irreducible
$\mathbb{F}_4G$-modules with dimension $18$ 
that cannot be realised over $\mathbb{F}_2$ (see~\cite{BAtlas}).

\begin{table}[!ht]
\renewcommand{\baselinestretch}{1.1}\selectfont
\centering
\begin{tabular}{ l l l l }
\hline
 $G$ &  $p$ & $\dim_{\mathbb{F}_p}(V)$ & $b(G)$   \\
\hline 
$\M_{11}$ & $2$ & $10$ & $2$ \\
& $3$ & $5^{(2)}$,  $10$ & $2$ \\
$\M_{12}$ & $2$ & $10$ & $3$ \\
& $3$ & $10^{(2)}$ & $2$ \\
$\M_{12}{:}2$ & $2$ & $10$ & $3$ \\
$2.\M_{12}$ & $3$ & $6^{(2)}$ & $3$ \\
& & $10^{(2)}$ & $2$ \\
$2.\M_{12}.2^+$ & $3$ & $12$, $10^{(4)}$ & $2$ \\
$2.\M_{12}.2^-$ & $3$ & $12$ & $2$\\
$\M_{22}$ & $2$ & $10^{(2)}$ & $3$ \\
$\M_{22}{:}2$ & $2$ & $10^{(2)}$ & $3$\\
$3.\M_{22}$ & $2$ & $12^\sharp$ & $3$ \\
$\M_{23}$ & $2$ & $11^{(2)}$ & $3$ \\
$\M_{24}$ & $2$ & $11^{(2)}$ & $3$ \\
$\J_1$ & $2$ & $20$ & $2$ \\
$\J_2$ & $2$ & $12^\sharp$ & $2$\\
$\J_2{:}2$ & $2$ & $12$ & $2$\\
$2.\J_2$ & $3$ & $12^\sharp$, $14$ & $2$ \\
& $5$ & $6$ & $2$ \\
$2.\J_2.2^+$ & $3$ & $12$ & $2$ \\
$2.\J_2.2^-$ & $3$ & $12,14^{(2)}$ & $2$ \\
$3.\J_3$ & $2$ & $18^\sharp$ & $2$ \\
   \hline
\end{tabular}
\quad
\begin{tabular}{ l l l l }
\hline
 $G$ &  $p$ & $\dim_{\mathbb{F}_p}(V)$  & $b(G)$  \\
\hline 
$\HS$ & $2$ & $20$ & $2$ \\
$\HS{:}2$ & $2$ & $20$ & $2$ \\
$\McL$ & $2$ & $22$ & $2$ \\
& $3$ & $21$ & $2$ \\
$\McL{:}2$ & $2$ & $22$ & $2$\\
& $3$ & $21^{(2)}$ & $2$ \\
$\Ru$ & $2$ & $28$ & $2$ \\
$2.\Suz$ & $3$ & $12$ & $3$\\
$2.\Suz.2^+$ & $3$ & $12^{(2)}$ & $3$ \\
$2.\Suz.2^-$ & $3$ & $24^\sharp$ & $2$\\
$3.\Suz$ & $2$ & $24^\sharp$ & $2$ \\
$6.\Suz$ & $7$ & $12^{(2)}$  & $2$\\
& $13$ & $12^{(2)}$ & $2$ \\
$\Co_3$ & $2$ & $22$ & $2$ \\
& $3$ & $22$ & $2$ \\
$\Co_2$ & $2$ & $22$ & $3$ \\
& $3$ & $23$ & $2$ \\
$\Co_1$ & $2$ & $24$ & $3$ \\
$2.\Co_1$ & $3$ & $24$ &  $2$ \\
& $5$ & $24$ & $2$ \\
& $7$ & $24$ & $2$ \\
& & & \\
\hline
\end{tabular}
\caption{$\mathbb{F}_pG$-modules $V$ on which $G$ has no regular orbit}
\label{tab:totalex}
\end{table}

A finite primitive permutation group $X$ is of  \textit{affine type} if its socle $V$ is an $\mathbb{F}_p$-vector space for some  prime $p$, in which case $X=V{:}X_0$ and $V$ is a faithful irreducible $\mathbb{F}_pX_0$-module, 
 where $X_0$ denotes the stabiliser of the vector $0$ in $X$. Now $b(X)=b(X_0)+1$,  
so classifying the primitive permutation groups of affine type with a  base of size 2 
amounts to determining which finite   groups $G$, primes $p$, and faithful 
irreducible  $\mathbb{F}_pG$-modules $V$  are 
such that $G$ has a regular orbit on~$V$.
Thus, as an immediate consequence of Theorem~\ref{regular}, we  obtain the following.

\begin{cor}
Let $X$ be a primitive permutation group of affine type with socle $V\simeq \mathbb{F}_p^d$ where $p$ is a prime dividing $|X_0|$ and $X_0$ is a covering group of an almost simple group whose socle is  sporadic. If $b(X)>2$, then $(G,p,d)$ is    listed in
Table $\ref{tab:totalex}$ where $G=X_0$, and $b(X)=b(G)+1$. In particular, $b(X)\leq 4$.
\end{cor}

The proof of Theorem \ref{regular} proceeds as follows. If $G$ has no regular orbit on $V$, then the dimension of $V$ is bounded above by some  integer  $u(G,p)$ (see Lemma \ref{general bound}). If $u(G,p)$ is less than the minimal dimension $m(G',p)$ of a faithful irreducible representation of the derived subgroup $G'$ in characteristic $p$ (given by \cite{Jan2005}), then we have a contradiction.  Otherwise, in most cases, the $p$-modular Brauer character table of $G$ is known, so we can determine the possible dimensions for $V$. 
We then use a variety of computational techniques in  {\sf GAP}~\cite{GAP4} and 
{\sc Magma}~\cite{Magma} to 
determine the base size $b(G)$ of $G$ in its action of $V$. For those cases where the $p$-modular Brauer character table of $G$ is not known---namely  when $(G,p)$ is one of $(\J_4,2)$, $(\Co_1,2)$, $(2.\Co_1,3)$ or  $(2.\Co_1,5)$---we use other methods to determine the possible dimensions for $V$  (see \S\ref{s:dims}).

The upper bound $u(G,p)$ is defined in terms of a well-known parameter.
Let $G$ be an almost simple group with socle $T$, and for each non-trivial $g\in G$, define $r(g)$ to be the minimal number of $T$-conjugates of $g$ that generate $\langle T,g\rangle$.  When $T$ is sporadic,  upper bounds on $r(G):=\max{\{r(g):g\in G\setminus \{1\}\}}$ were determined 
in \cite{GurSax2003} (see \cite[Table 1]{GurSax2003} and the proof 
of \cite[Lemma 7.6]{GurSax2003}), but these are not always sufficient for our purposes. To determine  the best possible bound on the dimension of $V$,  we compute the exact values of the  $r(g)$ and record these in the following theorem. In particular, this result considerably improves the upper bounds of \cite{GurSax2003} on $r(G)$ and may be of independent interest.

\begin{thm}
\label{thm:rg}
Let $G$ be an almost simple group whose socle  is sporadic, 
and let $g\in G$ be non-trivial.
Either $g^2=1$ and $r(g)=3$, or $g^2\neq 1$ and $r(g)=2$, 
or the class name of $g$ is listed in Table~$\ref{tab:rg}$.
\end{thm}

\begin{table}[!ht]
\renewcommand{\baselinestretch}{1.1}\selectfont
\centering
\begin{tabular}{ c | c c c c }
\hline
$G$ &  $r(g)=3$ & $r(g)=4$ & $r(g)=5$ & $r(g)=6$ \\
\hline
$\M_{22}({:}2)$ && 2B &&\\
$\J_2({:}2)$ & 3A & 2A &&\\
$\HS({:}2)$ & 4A & 2C &&\\
$\McL({:}2)$ & 3A &&&\\
$\Suz({:}2)$ && 3A &&\\
$\Co_2$ && 2A &&\\
$\Co_1$ & 3A &&&\\
$\Fi_{22}({:}2)$ & 3A, 3B & 2D && 2A\\
$\Fi_{23}$ & 3A, 3B &&& 2A\\
$\Fi_{24}'({:}2)$ & 3A, 3B && 2C &\\
$\HN({:}2)$ & 4D &&&\\
$\Ly$ & 3A &&&\\
$\Baby$ && 2A &&\\
\hline
\end{tabular}
\caption{Exceptional values of $r(g)$} 
\label{tab:rg}
\renewcommand{\baselinestretch}{1.3}\selectfont
\end{table}	

This paper is organised as follows. 
In \S\ref{s: prelim} we collect some notation, definitions and basic facts. 
In \S\ref{s: bounds} we determine bounds both for the dimensions of faithful irreducible representations admitting 
no regular orbit and for base sizes.  
In \S\ref{s:dims} we address the ``dimension gaps" that occur when $m(G',p)\leq u(G,p)$ and the $p$-modular Brauer character table of $G$ is not known.
In \S\ref{s: comp} we briefly discuss computational aspects, and 
in \S\ref{s:rg} and \S\ref{s:proofs} we prove Theorems~\ref{thm:rg} and~\ref{regular},  respectively.

\section{Preliminaries}
\label{s: prelim}

Let  $G$ be a finite group. We denote the derived subgroup of $G$ by $G'$, the centre of $G$ by $Z(G)$, and the conjugacy class of $g\in G$ by $g^G$.  A finite group $G$ is \textit{almost simple} if $T\unlhd G\leq \Aut(T)$ for some non-abelian simple group $T$. The subgroup of $G$ generated by the minimal normal subgroups of $G$ is the \textit{socle} of $G$, and $G$ is almost simple precisely when its socle is a non-abelian simple group. 
If $G=G'$ and $G/Z(G)$ is simple, then $G$ is \textit{quasisimple}. 
  
A finite group $L$ is a  \textit{covering group}  of  $G$ (by a group $M$) if  
$L/M\simeq G$ where  $M\leq Z(L)\cap L'$.  
Every finite group $G$ has a  \textit{universal covering group} \cite[Theorem 4.226]{Gor1982}, 
which is a covering group of $G$ by $M(G)$, the Schur multiplier of $G$ (see \cite[\S 4.15]{Gor1982} for a definition). 
If $L$ is a covering group of $G$ by a group $M$, then $L$ is a homomorphic image of 
some universal covering group of $G$ and $M$ is a homomorphic image 
of $M(G)$ (see \cite[Proposition 4.227]{Gor1982}). 
Universal covering groups are determined up to isoclinism in general, 
and up to isomorphism when $G$ is simple \cite[Chap.~4, \S 1]{Atlas}. 
A finite group $L$ is a covering group of 
an almost simple group $G$ precisely when $L/Z(L)\simeq G$ and $Z(L)\leq L'$.

Let $T$ be one of the $26$ sporadic simple groups. A wealth of information about these groups may be found in \cite{Atlas}, with which our notation is (more or less)  consistent.  It is well known that $M(T)$ is cyclic (see, for example, \cite[Theorem 5.1.4]{KleLie1990}), and also that $\Aut(T)=T$ or $T{:}2$. 
If $\Aut(T)=T{:}2$, then $M(\Aut(T))\leq M(T)$. 
(To see this, observe that the derived subgroup of a covering
group of $T{:}2$ is perfect with central quotient $T$.)
Thus by~\cite{Atlas},  $M(\Aut(T))=C_s$ where $s=(2,|M(T)|)$, and if $s=2$, then $\Aut(T)$ has  exactly two universal covering groups. The ordinary character table of one of these groups is listed in \cite{Atlas}, and  we denote this group by $2.T.2^+$. We denote the other universal covering group by $2.T.2^-$; its character table is easily derived from that of $2.T.2^+$  (see \cite[Chap.\ 6, \S 6]{Atlas}).
In Table \ref{tab:sporadic}, for the convenience of the reader, we list 
the orders of $T$, its Schur multiplier $M(T)$, 
and its outer automorphism group $\Out(T)$.

\begin{table}[!ht]
\renewcommand{\baselinestretch}{1.1}\selectfont
\centering
\begin{tabular}{ c c c  c  }
\hline
$T$ & \!$|M(T)|$\! & \!$|\Out(T)|$\!  & $|T|$ \\
\hline
$\M_{11}$ & $1$ & $1$   & $2^4.3^2.5.11$\\
$\M_{12}$ & $2$ & $2$   & $2^6.3^2.5.11$\\
$\M_{22}$ & $12$ & $2$ &  $2^7.3^2.5.7.11$\\
$\M_{23}$ & $1$ & $1$  & $2^7.3^2.5.7.11.23$\\
$\M_{24}$ & $1$ & $1$ &  $2^{10}.3^3.5.7.11.23$\\ 
$\J_1$ & $1$ & $1$ &  $2^3.3.5.7.11.19$\\
$\J_2$ & $2$ & $2$  & $2^7.3^3.5^2.7$\\
$\J_3$ & $3$ & $2$ &  $2^7.3^5.5.17.19$\\
$\J_4$ & $1$ & $1$ &  $2^{21}.3^3.5.7.11^3.23.29.31.37.43$\\
$\HS$ & $2$ & $2$ &  $2^9.3^2.5^3.7.11$\\
$\McL$ & $3$ & $2$ & $2^7.3^6.5^3.7.11$\\
$\He$ & $1$ & $2$ &  $2^{10}.3^3.5^2.7^3.17$\\
$\Ru$ & $2$ & $1$ &  $2^{14}.3^3.5^3.7.13.29$\\
$\Suz$ & $6$ & $2$ &  $2^{13}.3^7.5^2.7.11.13$\\
$\ON$ & $3$ & $2$ &  $2^9.3^4.5.7^3.11.19.31$\\
$\Co_3$ & $1$ & $1$ & $2^{10}.3^7.5^3.7.11.23$\\
$\Co_2$ & $1$ & $1$ &  $2^{18}.3^6.5^3.7.11.23$\\
$\Co_1$ & $2$ & $1$ &  $2^{21}.3^9.5^4.7^2.11.13.23$\\
$\Fi_{22}$ & $6$ & $2$ &  $2^{17}.3^9.5^2.7.11.13$\\
$\Fi_{23}$ & $1$ & $1$ &  $2^{18}.3^{13}.5^2.7.11.13.17.23$\\
$\Fi_{24}'$ & $3$ & $2$ &  $2^{21}.3^{16}.5^2.7^3.11.13.17.23.29$\\
$\HN$ & $1$ & $2$ &  $2^{14}.3^6.5^6.7.11.19$\\
$\Ly$ & $1$ & $1$ & $2^8.3^7.5^6.7.11.31.37.67$\\
$\Th$ & $1$ & $1$ &  $2^{15}.3^{10}.5^3.7^2.13.19.31$\\
$\Baby$ & $2$ & $1$ &  $2^{41}.3^{13}.5^6.7^2.11.13.17.19.23.31.47$\\
$\Mon$ & $1$ & $1$ & 
 $2^{46}.3^{20}.5^9.7^6.11^2.13^3.17.19.23.29.31.41.47.59.71$\\
\hline
\end{tabular}
\caption{The sporadic simple groups $T$} 
\label{tab:sporadic}
\renewcommand{\baselinestretch}{1.3}\selectfont
\end{table}

For a prime $p$, the ($p$-modular) Brauer character table of $G$ encodes information about the absolutely irreducible  representations of $G$ in characteristic $p$ by lifting the eigenvalues of the matrices representing $G$ to a field of characteristic $0$ (see \cite[\S 4]{BAtlas} for a definition).
  We often use the known Brauer character tables of the sporadic simple groups.
 For those sporadic simple groups $T$ whose order is at most $|\McL|$, the 
Brauer Atlas \cite{BAtlas} contains the Brauer character tables of all 
bicyclic extensions of $T$ for primes $p$ dividing $|T|$; these tables are 
known for some larger groups, see \cite{mod-atlas} for the available data. 
 
   Let $F$ be a field. We denote the  group algebra of  $G$ over $F$ by $FG$. All $FG$-modules in this paper are finite-dimensional, and we denote the dimension or degree of an $FG$-module $V$ by $\dim_F(V)$. An  irreducible $FG$-module $V$ is \textit{absolutely irreducible} if the extension of scalars $V\otimes_F E$ is irreducible for every field extension $E$ of $F$, and this occurs precisely when $\End_{FG}(V)=F$ 
(see \cite[Lemma VII.2.2]{BlaHup1981}), 
where $\End_{FG}(V)$ denotes the set of $FG$-endomorphisms of $V$. We denote the finite field of order $q$ by $\mathbb{F}_q$.
 
 Let $V$ be an irreducible $\mathbb{F}_pG$-module where $p$ is prime, and let $k:=\End_{\mathbb{F}_pG}(V)$. Now $k$ is a finite division ring and therefore a field, so $V$ is an absolutely irreducible $kG$-module where scalar multiplication is evaluation. 
Let $\chi$ be the 
Frobenius
character of $V$ as a $kG$-module, and let $H$ be the Galois group of the field extension $k/\mathbb{F}_p$. By \cite[Theorem VII.1.16]{BlaHup1981}, $k=\mathbb{F}_p(\{\chi(g):g\in G\})$ and $V\otimes_{\mathbb{F}_p}k=\bigoplus_{\gamma\in H} V_\gamma$, where 
the $V_\gamma$ are pairwise non-isomorphic
absolutely irreducible $kG$-modules with character 
${}^\gamma\chi: G\rightarrow k:g\mapsto\gamma(\chi(g))$ 
which cannot be realised over any proper subfield of $k$. Now $\dim_k(V)$ is given by the $p$-modular Brauer character table of $G$, and we can use this table to determine $\{\chi(g):g\in G\}$ (see \cite[\S\S 2-5]{BAtlas}) and therefore the $\mathbb{F}_pG$-module $V$.

Let $N$ be a subgroup of $G$ with index $2$. Let $V$ be an irreducible $FG$-module, 
and let $W$ be  an irreducible $FN$-submodule of $V|_N$, the restriction 
of $V$ to $N$. It is well known from Clifford theory that  either 
$V|_N=W$, or $V|_N=W\oplus Wg$ for all $g\in G\setminus N$. 
We frequently use the following observations without reference. 
If $V|_N=W\oplus Wg$ and $W$ is a faithful $FN$-module, then the base sizes of $N$ on $W$ and $Wg$ are equal,  and this base size is at least $b(G)$
(since $N\neq 1$). If instead $V|_N=W$ and $V$ is a faithful $FG$-module, then $b(N)\leq b(G)$.

\section{Some useful bounds}
\label{s: bounds}

If $G$ acts faithfully on a finite set $\Omega$, 
then $|G|\leq |\Omega|^{b(G)}$ since, for every base  $B$ of $G$, 
each $g\in G$ is uniquely determined by $\{\alpha^g : \alpha\in B\}$.  
Thus we have the following elementary but useful result.

\begin{lemma}
\label{lowerbound}
Let $G$ be a finite group and $F$ a finite field.
If $V$ is a faithful $FG$-module, 
then $|G|\leq (|V|-1)^{b(G)}$.
\end{lemma}

For a group $G$, field $F$ and $FG$-module $V$, define   
$C_V(g):=\{v\in V:vg=v\}$ for all $g\in G$.  
The following generalises \cite[Lemma 3.3]{FawObrSax2016}.
 
\begin{lemma}
\label{upperbound}
\label{strong bound}
Let $G$ be a finite group and $F$ a finite field. 
Let $V$ be a faithful $FG$-module. 
Let $X$ be a set of representatives for the conjugacy classes of
elements of prime order in~$G$. If $n$ is a positive integer for which
 \begin{equation}
 \label{eqn:strong}
 |V|^n>\sum_{g\in X}|g^{G}||C_{V}(g)|^n,
\end{equation}
then $b(G)\leq n$.   
\end{lemma}

\begin{proof}
 Let $Y$ be the set of elements of prime order in $G$, and let $W$ be a faithful $FG$-module.  If $G$ has no regular orbit on $W$, then $W=\bigcup_{g\in Y}C_W(g)$, so $|W|\leq \sum_{g\in X}|g^{G}||C_{W}(g)|$. The base size of $G$ on $V$ is at most $n$ if and only if $G$ has a regular orbit on the faithful $FG$-module $V^n$. Since $C_{V^n}(g)=C_V(g)^n$, the result follows.
\end{proof}

Let $G$ be a finite group such that $G/Z(G)$ is almost simple  
with socle $T$, and let $g\in G\setminus Z(G)$. Now $\langle T,Z(G)g\rangle$ is generated by  the $T$-conjugates of $Z(G)g$, so we may define $r(g)$ to be  the minimal number of $T$-conjugates of $Z(G)g$   generating  $\langle T,Z(G)g\rangle$. This extends the definition given in \S\ref{s:intro}.  The following generalises  \cite[Lemma 3.5]{FawObrSax2016}.

\begin{lemma}
\label{general bound}
Let $G$ be a finite  
group with $G/Z(G)$ almost simple and $V$ a faithful irreducible 
$\mathbb{F}_qG$-module where $q$ is a prime  power. Let $X$ be a set of representatives for the conjugacy classes of  non-central elements of prime order in~$G$, and let $u(G,q)$ be the largest integer such that
$$ 1\leq \sum_{g\in X}|g^G|(1/q)^{u(G,q)/r(g)}.  $$
If $G$ has no regular orbit on $V$, then  
$ \dim_{\mathbb{F}_q}(V)\leq u(G,q). $
\end{lemma}

\begin{proof}
Let $d:=\dim_{\mathbb{F}_q}(V)$. 
Note that  $C_V(g)=\{0\}$ for $g\in Z(G)\setminus \{1\}$.  
If $G$ has no regular orbit on $V$, then by 
Lemma \ref{strong bound} and \cite[Lemma 3.4]{FawObrSax2016},
$$q^d\leq \sum_{g\in X}|g^G||C_V(g)|\leq \sum_{g\in X}|g^G|q^{d-d/r(g)},$$ 
so $d\leq u(G,q)$. 
\end{proof}

\section{Dimension gaps} \label{s:dims}

In this section we consider those cases where the upper bound
$u(G,p)$ for the dimension of a faithful irreducible $\mathbb{F}_pG$-module
on which $G$ has no regular orbit (as given by Lemma~\ref{general bound})
is at least  the minimal dimension of a  
faithful irreducible representation of $G'$ in characteristic $p$
(as given by \cite{Jan2005}), but the $p$-modular
Brauer character table of $G$ is not yet known. 
These ``dimension gaps" occur when 
$(G,p)$ is one of $(\J_4,2)$,  $(\Co_1,2)$, $(2.\Co_1,3)$ or $(2.\Co_1,5)$. 
For each, we first compute $u(G,p)$ using the ordinary 
character table of $G$ and Theorem~\ref{thm:rg},
and then determine the representations 
whose dimensions are at most $u(G,p)$.
Note that these results (for the dimension bound of $250$) 
are stated in \cite{HissMalle2001}, but explicit proofs, 
 which often depend on computations with the {\sc MOC} 
 system \cite{HissJanLuxPar}, are omitted.

\begin{lemma}
\label{lemma:Co1}
There is a unique faithful irreducible $2$-modular representation of $\Co_1$ of 
degree at most $u(\Co_1,2)=117$. This representation has degree $24$. 
\end{lemma}

\begin{proof}
Recall from the list of maximal subgroups in \cite{Atlas} (as corrected in \cite{BAtlas}) that  $\Co_1$ can be generated by subgroups $3^6{:}2.\M_{12}$
and $3.\Suz{:}2$, intersecting in $3^5{:}(2\times \M_{11})$.
The $2$-modular irreducibles of $3.\Suz{:}2$ of degree at most $117$ 
have degrees $1$ and $24$. 
Hence every irreducible character of $\Co_1$ of degree at most $117$
has character restriction composed of $1$ and $24$. 
Restricting to $3^5{:}(2\times \M_{11})$, the $24$ remains irreducible.
Since the orbits of $2.\M_{12}$ on the linear characters of $3^6$ are
$1+24+264+440$, the only possible character restriction to 
$3^6{:}2.\M_{12}$ is $24^k+1^j$. 

The $24$ is not in the principal block of either $3.\Suz{:}2$, 
or of $3^6{:}2.\M_{12}$. 
Hence the principal and non-principal block
components of the restrictions to these subgroups coincide. 
If $j>0$, then the perfect group $3^6{:}2.\M_{12}$ acts trivially
on the non-zero principal block component, a contradiction. 
Hence $j=0$. Since the irreducible $24$ for $3.\Suz{:}2$ 
has no non-splitting self-extensions, $3.\Suz{:}2$ acts completely 
reducibly, so there is a $24$-dimensional subspace 
invariant under $\Co_1$, implying $k=1$.
Finally, for every irreducible  representation of degree $24$,
there is a unique amalgamation inside $\GL_{24}(2)$, so there
is a unique such representation of $\Co_1$. 
\end{proof}

\begin{lemma}
\mbox{}
 \label{lemma:2Co1}
\begin{enumerate}
\item There is a unique faithful irreducible $3$-modular representation
of $2.\Co_1$ of degree at most $u(2.\Co_1,3)=74$.  
This representation has degree $24$.
\item There is a unique faithful irreducible $5$-modular representation
of $2.\Co_1$ of degree at most $u(2.\Co_1,5)=50$. 
This representation has degree $24$.
\end{enumerate}
\end{lemma}

\begin{proof}
In each case, we can construct such a representation from 
two copies of $2^{12}{:}\M_{24}$, intersecting in $2^{6+12}{:}3.S_6$. 
Since the orbits of $\M_{24}$ on the linear characters of $2^{12}$ are
$1+24+276+1771+2024$, the only faithful 
irreducible representation of $2^{12}{:}\M_{24}$ 
of degree less than $276$
in odd characteristic is the monomial $24$. This remains irreducible on
restriction to $2^{6+12}{:}3.S_6$. 

Moreover, every relevant faithful irreducible representation of $2.\Co_1$
must have character restriction $24^k$ to each of these subgroups.
Since the irreducible $24$ for $2^{12}{:}\M_{24}$ has no
non-splitting self-extensions, the subgroups in question act
completely reducibly, implying that $k=1$.
Since every matrix that commutes with the action of
$2^{6+12}{:}3.S_6$ commutes with the action of both copies of 
$2^{12}{:}\M_{24}$,
there is a unique amalgamation of groups into $\GL_{24}(3)$, or 
$\GL_{24}(5)$, and so a unique representation of $2.\Co_1$ of dimension~$24$.
\end{proof}

Our proof for $(\J_4,2)$ is 
motivated by the approach of \cite[\S4.3.21]{Jan2005},
which in turn is based on \cite[Chap.~6]{Ben1981}.

\begin{lemma} \label{lemma:J4}
There is a unique faithful irreducible $2$-modular representation of $\J_4$ 
of degree at most $u(\J_4,2)=129$. This representation has degree $112$.
\end{lemma}

\begin{proof}
Let $V$ be a faithful irreducible $2$-modular representation of $G:=\J_4$ 
of degree at most $129$. Recall 
that $G$ can be generated by subgroups
$K:=U_3(11){:}2$ and $L:=11^{1+2}{:}(5 \times 2.S_4)$, intersecting in 
$H:=11^{1+2}{:}(5 \times 8{:}2)$. Also, $K'=U_3(11)$ and $L$ generate $G$
and intersect in $H_1:=11^{1+2}{:}(5 \times 8)$;
see \cite[p.~63]{Ben1981}.
In particular, there exists an involution
$z\in H\setminus H_1$ where $z\in L$ and $K=\langle K',z\rangle$.

The $2$-modular Brauer character table of $K$ shows that $V$ has
character restriction $110+1^k$ to $K$. Similarly, $V$ has
character restriction $110'+\lambda$ to $L$, where $110'$ is one of the 
three irreducible Brauer characters of 
this degree, and $\lambda$ 
has constituents of degree $1$ and $2$.
By considering the unique conjugacy class of elements of order 3 in $G$,
we deduce that $k\geq 2$, the irreducible Brauer
character of degree $110$ occurring in the restriction to $L$
is one of the non-rational ones, and $\lambda$ has a unique
constituent of degree $2$ apart from $k-2$ linear ones. 

By considering the unique conjugacy class of elements of order 5 in $G$,
we deduce that the constituents of $\lambda$ are necessarily
rational, so are uniquely determined by their degree.
The $2$-modular Brauer character table of $H$ shows that 
the irreducible Brauer characters of degree $110$ of $K$ and $L$
restrict irreducibly to $H$. A consideration of $2$-blocks yields 
the block decomposition of the various restrictions as 
$$ V|_K\sim (110+1^k)
\quad \text{and} \quad 
V|_L\sim 110'\oplus(2+1^{k-2}) 
\quad \text{and} \quad 
V|_H\sim 110\oplus(1^k) .$$
 
Hence there is a unique $L$-submodule $U_{110}$ of $V$ of dimension $110$, 
which cannot be $K$-invariant, so the unique
$110$-local $K$-submodule of $V$ is reducible.
(Here we freely use terminology from \cite{LMR}.)
By computing suitable cohomology groups, we determine 
all possible downward $K$-extensions of $110$ with kernel
having only trivial constituents. This shows that the local submodule in 
question, say $U_{111}$, is uniserial with descending composition
series $110/1$. 
In particular, this specifies a trivial $K$-submodule $U_1$ of $V$.
Moreover, $(V/U_{111})|_K\sim 1^{k-1}$; thus $K'$, being perfect,
acts trivially on this quotient. 

Observe that $[L{:}H]=3$,
where the action of $L$ on the cosets of $H$ is equivalent to 
the natural action of $S_3\cong S_4/V_4$; thus the associated
permutation module is semisimple of shape $2\oplus 1$.
Since $U_1$ cannot be $L$-invariant, 
the $L$-submodule $U_1^L$ generated by $U_1$
is either irreducible of degree $2$, or semisimple of shape $2\oplus 1$.
Hence we get an $L$-submodule $110'\oplus U_1^L$, which contains $U_{111}$,
and so is $K'$-invariant. Hence $110'\oplus U_1^L=V$,
implying that $\dim_{\mathbb{F}_2}(V)\in\{112,113\}$;
in other words, $k=2$ or $k=3$.

Let $X\leq V$ be the image of the action of $1+z$,
where we view the latter as an element of
both $\mathbb{F}_2K$ and $\mathbb{F}_2L$.
Since $1+z$ has an image of dimension $55$ on $110$,
both cases for $V|_L$ yield $\dim_{\mathbb{F}_2}(X)=56$, where
$X$ intersects non-trivially with the irreducible direct summand $2$.

Next, again using a cohomological approach, we determine all indecomposable
$K$-modules (up to isomorphism) having the constituent $110$ once, and
the trivial constituent $1$ with multiplicity at most $3$, where we 
restrict to those modules having $110$ neither in their head nor 
in their socle. This yields two isomorphism types $V_{112}$ and $V_{112'}$
of modules, both uniserial with descending composition series $1/110/1$,
a module $V_{113}$ of dimension $113$ with head of shape $1\oplus 1$
and socle of shape $1$, and the dual $V_{113}^\ast$ of $V_{113}$. 
In all cases the unique $110$-local
submodule is isomorphic to $U_{111}$.

Suppose that $k=3$. If $V|_K = V_{113}$, for which $1+z$ 
has image of dimension $56$, then $X\leq U_{111}$.
This implies that $110'\oplus 2$ contains $U_{111}$. Since 
the latter coincides with the radical of $V_{113}$, we conclude
that $110'\oplus 2$ is $K'$-invariant, a contradiction.
Similarly, if $V|_K = V_{113}^\ast$, then, since $V|_L$ is self-dual,
we obtain a contradiction by dualising the picture.
Finally, let $V|_K$ have a direct summand isomorphic to $V_{112}$
or $V_{112'}$. Now $1+z$ has an image of dimension $55$ 
on $V_{112'}$, say, and $1+z$ has an image of dimension $56$ 
on $V_{112}$. Hence $V|_K=V_{112}\oplus 1$, where again
$X\leq U_{111}$, yielding a contradiction as above.
This excludes the case $k=3$. 

Hence $k=2$, so $\dim_{\mathbb{F}_2}(V)=112$. By the above, $V|_K=V_{112}$.
The arguments of \cite[pp.\ 63ff.]{Ben1981}
now imply that $V$ is uniquely determined up to isomorphism.
\end{proof}

\section{Comments on computations} \label{s: comp}

We usually used the 
{\sc Atlas} database \cite{web-atlas} 
to access explicit matrix and permutation representations 
on standard generators, 
and straight line programs on these for conjugacy class representatives.
In {\sf GAP} we accessed this data 
through the {\sc AtlasRep} package \cite{AtlasRep}.
The ordinary and Brauer character tables from
\cite{Atlas,BAtlas,mod-atlas} are available through the
Character Table Library \cite{CTblLib} of {\sf GAP}.

We used the {\sc Orb} package \cite{ORB} available 
through {\sf GAP}. It has highly optimised techniques 
to enumerate orbits of vectors or subspaces in a $G$-module $V$.
It can be used directly to enumerate a $G$-orbit point-by-point.
But a critical feature is that it can also enumerate a $G$-orbit 
in larger pieces consisting of suborbits with respect to a helper 
subgroup $U\leq G$.  During the enumeration process, to recognise quickly
whether a $U$-suborbit has been encountered before, helper $U$-sets,
homomorphic images of the given ones, are used; 
for example, if the $G$-orbit consists of vectors in a 
$G$-module $V$, then the helper $U$-set may consist 
of the vectors in an epimorphic image of the $U$-module $V|_U$.
To fully utilise the helper $U$-sets, these must be enumerated
in turn, which is done using the same process, 
giving rise to a divide-and-conquer strategy.
A detailed account of this approach is given in \cite{MNW},
whose terminology we freely borrow here. 

{\sc Magma} has an implementation of 
an algorithm of Cannon and Holt \cite{exp-math} which 
constructs faithful irreducible representations 
defined over a given finite field of a finite permutation group; 
we used this to construct representations,
either all or those of specified degree, of certain small degree 
permutation groups. Occasionally, we used our implementation in {\sc Magma} 
of the algorithm of \cite{GlaLeeOBr2006} to conjugate a given representation 
to one defined over a subfield.

Applications of Lemma \ref{strong bound} 
require knowledge of conjugacy classes of $G$.
To compute these, we sometimes used the infrastructure of
\cite{BaaHolLeeOBr2015} available in {\sc Magma};
classes in $J_4$ were written down directly using the results of
\cite{Janko1976} as summarised at \cite{web-atlas}.

\section{Proof of Theorem \ref{thm:rg}}
\label{s:rg}

Let $G$ be an almost simple group whose socle $T$ is sporadic. 
First suppose that $T\neq \Mon$; we address this case in Lemma \ref{Monster}.

\begin{enumerate}
\item[(i)] 
Using explicit words given on standard generators from the {\sc Atlas} database, 
or the general purpose algorithm available in {\sf GAP}, we determine
representatives of conjugacy classes of $G$.

\item[(ii)] 
For each class representative $g$, 
we perform a random search through 
$g^G$ for a subset $S$ generating $G$ or $T$. 
If $G$ has a ``small degree" permutation representation,
then we check generation by $S$ directly.
For $\J_4,\,\HN,\,\Ly,\,\Th$ and $\Baby$, we use instead a 
faithful matrix representation and a different generation check:
we select a set of primes 
whose product divides the order of $T$, but of none of its maximal 
subgroups, and now search randomly 
in $\langle S \rangle$ for elements having these orders.
Hence, for each class representative $g$, we obtain an upper bound
$u(g)$ to $r(g)$. 

\item[(iii)]
Clearly $r(g)\geq 3$ if $g$ is an involution,
and $r(g)\geq 2$ otherwise. If $u(g)$ equals this lower bound,
then $r(g) = u(g)$.
This leaves unresolved the cases listed in Table~$\ref{tab:rg}$. 
(Since in all cases $u(g) = r(g)$, the random search
achieved the best possible outcome.)

\item[(iv)] 
Since no non-trivial element of $\Aut(G)$ stabilises a generating
$\ell(g)$-tuple of distinct elements of $g^G$, we deduce that 
$$ \prod_{i=0}^{\ell(g) - 1}(|g^G|-i)\geq|{\Aut(G)}|.$$
This provides a new lower bound $\ell(g)$ for $r(g)$, and resolves
the following cases where $\ell(g)=u(g)$:  
$$ 
(\J_2,\text{3A}),\, 
(\McL,\text{3A}),\, 
(\Co_1,\text{3A}),\, 
(\Fi_{22},\text{3A}),\, 
(\Fi_{23},\text{3A}),\, 
$$
$$
(\Fi_{24}',\text{3A}),\, 
(\Fi_{24}',\text{3B}),\, 
(\Fi_{24}'{:}2,\text{2C}),\, 
(\HN{:}2,\text{4D}),\, 
(\Ly,\text{3A}),\, 
(\Baby,\text{2A}).
$$

\item[(v)]
In most cases, a search through a set of 
representatives $g_1,\ldots,g_{u(g)-1}$ 
of the $G$-orbits on the set of $(u(g)-1)$-tuples of $g^G$ is feasible.
These are found readily as follows. Fixing $g_1:=g$, we let $g_2$ run through 
a set of representatives of the $C_G(g_1)$-orbits in $g^G$, for fixed $g_2$ we 
let $g_3$ run through a set of representatives of the $C_G(g_1,g_2)$-orbits 
in $g^G$, and so on.  We check directly the order of the 
subgroup generated by each tuple. This resolves the cases 

$$
(\M_{22}{:}2,\text{2B}),\, 
(\J_2,\text{2A}),\, 
(\HS,\text{4A}),\, 
(\HS{:}2,\text{2C}),\, 
(\Suz,\text{3A}),\, 
$$
$$
(\Co_2,\text{2A}),\, 
(\Fi_{22},\text{2A}),\, 
(\Fi_{22}{:}2,\text{2D}),\, 
(\Fi_{23},\text{2A}).
$$
\end{enumerate}

We now resolve the remaining cases.
\begin{lemma}
\label{lemma:fi22fi23}
\mbox{}
\begin{enumerate}
\item 
Let $G=\Fi_{22}$. If $g\in G$ is in class $\mathrm{3B}$, then $r(g)=3$.
\item 
Let $G=\Fi_{23}$. If $g\in G$ is in class $\mathrm{3B}$, then $r(g)=3$.
\end{enumerate}
\end{lemma}

\begin{proof}
In each case we know from (ii) that $r(g)\leq 3$. 
Hence we must show that there is no $h\in g^G$ such that 
$\{g,h\}$ generates $G$. It suffices
to let $h$ run through a set of representatives of the 
$C_G(g)$-orbits on $g^G$. Moreover, if $\{g,h\}$ generates $G$,
then $C_G(g)\cap C_G(h)=Z(G)=\{1\}$, so $h$ belongs
to a regular $C_G(g)$-orbit. Thus it suffices to find 
representatives of the regular $C_G(g)$-orbits on $g^G$.  
The latter are found, or their non-existence proved, 
by an application of {\sc Orb} with helper subgroups.
If there are relevant $C_G(g)$-orbits, then 
we determine the order of the subgroups thus generated, 
and verify that none is $G$. 

We summarise the details of our {\sc Orb} computations.

\begin{enumerate}
\item[(1)]
$G=\Fi_{22}$ has a faithful permutation representations on $3510$ points. 
The action of $G$ on its conjugacy class $g^G$,
which has size $25\,625\,600$, is equivalent to its action on the
cosets of $C_G(g)\cong 3^{1+6}.2^{3+4}.3^2$, where 
$N_G(\langle g\rangle)\cong 3^{1+6}.2^{3+4}.3^2.2$
is a maximal subgroup of $G$. 
There is a vector $v$ 
in the absolutely irreducible $\mathbb{F}_3G$-module $V$ of dimension $924$
which is fixed precisely by $C_G(g)$, so $v^G$ is equivalent to $g^G$ 
as a $G$-set.
Using the permutation character of the action of $G$ on the cosets of 
$C_G(g)$, we find that $C_G(g)$ has $64$ orbits in $v^G$. 
We use the chain of helper subgroups
$\{1\}=U_0<U_1=U_2<U_3=G$ specified in Table \ref{tbl:helpfi22},
where we also list the dimension $d_i$ of the various helper quotients of $V$.
In particular, we choose $U_1=U_2=C_G(g)$, but use distinct helper quotients.
To find all regular $C_G(g)$-orbits in the $G$-orbit $v^G$, 
we must enumerate at least $1-|C_G(g)|/|v^G|\sim 90\%$ of it.
Enumerating a total of $23\,471\,749$ vectors in $v^G$, that is $\sim91\%$
of $v^G$, we find $37$ orbits, precisely $5$ of which are regular 
$C_G(g)$-orbits.
Translating back to $g^G$, we find that none gives rise to a
two-element generating set of $G$, generating instead
either $G_2(3)$ or $A_9$.

\begin{table}[!ht]\renewcommand{\baselinestretch}{1.1}\selectfont
\centering
$\begin{array}{rrrrr}
\hline 
i & U_i & |U_i| & [U_i {:} U_{i-1}] & d_i \\
\hline
3 & \Fi_{22} & 64\,561\,751\,654\,400 & 25\,625\,600 & 924 \\
2 & 3^{1+6}.2^{3+4}.3^2 & 2\,519\,424 & 1 & 17 \\
1 & 3^{1+6}.2^{3+4}.3^2 & 2\,519\,424 & 2\,519\,424 & 5 \\
\hline
\end{array}$
\caption{Helper subgroups for $\Fi_{22}$}\label{tbl:helpfi22}
\end{table}

\item[(2)]
$G=\Fi_{23}$ has a faithful permutation representations on $31671$ points.
The action of $G$ on its conjugacy class $g^G$,
which has size $2\,504\,902\,400$, is equivalent to its action on the
cosets of $C_G(g)\cong 3^{1+8}.2^{1+6}.3^{1+2}.2A_4$, where 
$N_G(\langle g\rangle)\cong 3^{1+8}.2^{1+6}.3^{1+2}.2S_4$
is a maximal subgroup of $G$. 
There is a vector $v$ 
in the absolutely irreducible $\mathbb{F}_3G$-module $V$ of dimension $528$
which is fixed precisely by $N_G(\langle g \rangle)$, so $v^G$ has length 
$1\,252\,451\,200$, and is equivalent to the
action of $G$ on the cosets of $N_G(\langle g \rangle)$. 
Since a regular $C_G(g)$-orbit in
$g^G$ implies a $C_G(g)$-orbit in $v^G$ of length divisible by $|C_G(g)|/2$,
we must find representatives of the latter. Using the permutation character of the action of $G$ on the cosets of 
$N_G(\langle g \rangle)$, we find that $C_G(g)$ has $37$ orbits in $v^G$. 
We use the chain of helper subgroups
$\{1\}=U_0<U_1<U_2<U_3<U_4<U_5=G$
specified in Table \ref{tbl:helpfi23},
where we also list the dimension $d_i$ of the various helper quotients of $V$.
In particular, we choose $U_4=C_G(g)$, and 
use the same helper quotient for all the helper subgroups.
To find all relevant $C_G(g)$-orbits in the $G$-orbit $v^G$,
we must enumerate at least $1-|C_G(g)|/(2\cdot|v^G|)\sim 35\%$ of it.
Enumerating a total of $1\,157\,675\,328$ vectors in $v^G$, 
that is $\sim 90\%$ of $v^G$, 
we find $13$ orbits, none of which has length divisible by $|C_G(g)|/2$. \qedhere

\begin{table}[!htbp]\renewcommand{\baselinestretch}{1.1}\selectfont
\centering
$\begin{array}{rrrrr}
\hline 
i & U_i & |U_i| & [U_i {:} U_{i-1}] & d_i \\
\hline
5 & \Fi_{23} & 4\,089\,470\,473\,293\,004\,800 & 2\,504\,902\,400 & 528 \\
4 & 3^{1+8}.2^{1+6}.3^{1+2}.2A_4 & 1\,632\,586\,752 & 24 & 13 \\
3 & 3^{1+8}.2^{1+6}.3^{1+2} & 68\,024\,448 & 27 & 13 \\
2 & 3^{1+8}.2^{1+6} & 2\,519\,424 & 128 & 13 \\
1 & 3^{1+8} & 19\,683 & 19\,683 & 13 \\
\hline
\end{array}$
\caption{Helper subgroups for $\Fi_{23}$}\label{tbl:helpfi23}
\end{table}
\end{enumerate}
\end{proof}

\begin{lemma}\label{Monster}
If $g \in \Mon$ is in class $\mathrm{2A}$ or $\mathrm{2B}$ then $r(g)=3$; 
all other non-trivial elements satisfy $r(g) = 2$.
\end{lemma}
\begin{proof}
From the (almost complete) classification of the 
maximal subgroups of $\Mon$ (see \cite[Table 5.6]{wilson}), we deduce
that, for both $p=59$ and $71$,
there is a unique conjugacy class of maximal subgroups of order divisible
by $p$, the groups in question being isomorphic to $\text{L}_2(p)$. 
A consideration of class multiplication coefficients, 
computed from the ordinary character table of $\Mon$, shows
that these are non-zero for all conjugacy class triples 
$(\text{X},\text{X},\text{71A})$ and $(\text{X},\text{X},\text{59A})$,
where $\text{X}$ runs through all conjugacy classes except 
$\text{1A}$, $\text{2A}$, and $\text{2B}$.
Hence $r(g) = 2$ for all conjugacy classes containing elements 
$g$ of order at least $3$ and not fusing into both $\text{L}_2(59)$ 
and $\text{L}_2(71)$. 

This leaves the non-involutory conjugacy classes
$\text{3B}$, $\text{5B}$, and $\text{6E}$. Since the squares
of the elements of $\text{6E}$ belong to 
$\text{3B}$, it suffices to deal with the first two.
For each conjugacy class, $\text{X}$ say,
we compute the number of pairs of elements of $\text{X}$ 
whose product belongs to $\text{71A}$,
and compare this with the number of such pairs contained in some 
maximal subgroup isomorphic to $\text{L}_2(71)$. In each case 
there are (many) more pairs in $\Mon$ than are accounted for by these
maximal subgroups. This implies that $r(g)=2$ for 
elements $g$ belonging to either 
$\text{3B}$ or $\text{5B}$.

It remains to consider the involutory conjugacy classes 
$\text{2A}$ and $\text{2B}$.
The class multiplication coefficients associated 
with the conjugacy class triples $(\text{2B},\text{2B},\text{41A})$
and $(\text{2B},\text{41A},\text{71A})$ are both non-zero.
Similarly, the class multiplication coefficients 
associated with the conjugacy class triples $(\text{2A},\text{2A},\text{5A})$
and $(\text{2A},\text{5A},\text{71A})$ are both non-zero; moreover
conjugacy class $\text{5A}$ does not fuse into $\text{L}_2(71)$.
This implies that $r(g)=3$ for elements $g$ belonging to either 
$\text{2A}$ or $\text{2B}$.
\end{proof}

\section{Proof of Theorem \ref{regular}}
\label{s:proofs}

Let $T$ be a sporadic simple group, and let $G$ be a 
covering group of an almost simple group with socle $T$.
Let $V$ be a faithful irreducible $\mathbb{F}_pG$-module 
where $p$ is a prime dividing $|G|$. 
Let $k:=\End_{\mathbb{F}_pG}(V)$. 
Note that $p\nmid |Z(G)|$ since $Z(G)\leq k^*$. 

Suppose that $G$ has no regular orbit on $V$, so $b(G)>1$. 
Let $m(G,p)$ denote the minimal dimension of a faithful irreducible
representation of $G$ in characteristic $p$.
Let $u(G,p)$ be as defined in Lemma \ref{general bound}. Recall that $G'$ denotes the derived subgroup of $G$. Observe that 
$$ m(G',p) \leq m(G,p)\leq \dim_k(V)\leq \dim_{\mathbb{F}_p}(V)\leq u(G,p). $$

\begin{lemma}
\label{lemma:cases}
 $(G,p,\dim_{\mathbb{F}_p}(V))$ is listed in  Table $\ref{tab:cases}$.
\end{lemma}

\begin{proof}
We use the ordinary character table of $G$ and Theorem~\ref{thm:rg} to compute 
the upper bound $u(G,p)$. By \cite{Jan2005}, $m(G',p)$ is known. If $m(G',p)>u(G,p)$, then we have a contradiction. 
If $(G,p)$ is one of  $(\J_4,2)$, $(\Co_1,2)$, $(2.\Co_1,3)$ or $(2.\Co_1,5)$, 
then $\dim_{\mathbb{F}_p}(V)$ is determined  by 
Lemmas \ref{lemma:Co1}--\ref{lemma:J4}.
Otherwise,  
we use the  $p$-modular Brauer character table of $G$ 
to determine the possibilities for $\dim_{\mathbb{F}_p}(V)$. 
\end{proof}

We adopt several conventions in Table \ref{tab:cases}. The dimensions in bold are precisely those listed in Table \ref{tab:totalex}. We write $d^{(m)}$ when   there are exactly $m$ $d$-dimensional $\mathbb{F}_pG$-modules and $m>1$, except for the case $(G,p,\dim_{\mathbb{F}_p}(V))=(\M_{11},3,10)$. Here we write $\bf 10$ and $10^{(2)}$; the one in bold 
denotes the unique faithful irreducible $10$-dimensional $\mathbb{F}_3\M_{11}$-module 
with the property that  
an involution, viewed as an element of $\GL_{10}(3)$, has trace $-1\in \mathbb{F}_3$.

\begin{table}[!ht]
\renewcommand{\baselinestretch}{1.1}\selectfont
\centering
\begin{tabular}{ c | c c c c c c c }
\hline
$G$ & $p=2$ & $p=3$ & $p=5$ & $p=7$ & $p=11$ & $p=13$ & $p=23$ \\
\hline
$\M_{11}$ & $\bf 10$ & $\bf 5^{(2)}$, $\bf 10$, $10^{(2)}$ & $10$ \\
$\M_{12}$ & $\bf 10$, $32$ & $\bf 10^{(2)}$, $15^{(2)}$ & $11^{(2)}$ \\
$\M_{12}{:}2$ & $\bf 10$, $32$ & $20$ \\
$2.\M_{12}$ & & $\bf 6^{(2)}$, $\bf 10^{(2)}$ & $12$ \\
$2.\M_{12}.2^+$ & & $\bf 10^{(4)}$, $\bf 12$ & & & $10^{(4)}$ \\
$2.\M_{12}.2^-$ & & $\bf 12$, $20^{(2)}$ \\
$\M_{22}$ & $\bf 10^{(2)}$, $34$ & $21$ \\
$\M_{22}{:}2$ & $\bf 10^{(2)}$, $34$ & $21^{(2)}$ \\
$2.\M_{22}$ & & $20$ & & $10$ & $10^{(2)}$ \\
$2.\M_{22}.2^+$ & & $20^{(2)}$ & & $10^{(2)}$ & $10^{(4)}$ \\
$2.\M_{22}.2^-$ & & $20^{(2)}$ \\
$3.\M_{22}$ & $\bf 12$, $30$ \\
$\M_{23}$ & $\bf 11^{(2)}$, $44^{(2)}$ & $22$ \\
$\M_{24}$ & $\bf 11^{(2)}$, $44^{(2)}$ & $22$ \\
$\J_1$ & $\bf 20$ & & & & $7$ \\
$\J_2$ & $\bf 12$, $28$, $36$ & & $14$ \\
$\J_2{:}2$ & $\bf 12$, $28$, $36$ & & $14^{(2)}$ \\
$2.\J_2$ & & $\bf 12$, $\bf 14$ & $\bf 6$, $14$ & $12$ \\
$2.\J_2.2^+$ & & $\bf 12$ & $12$ & $12$ \\
$2.\J_2.2^-$ & & $\bf 12$, $\bf 14^{(2)}$ & $12$ & $12$ \\
$3.\J_3$ & $\bf 18$, $36^{(2)}$ \\
$\J_4$ & $112$ \\
$\HS$ & $\bf 20$ & $22$ & $21$ \\
$\HS{:}2$ & $\bf 20$ & $22^{(2)}$ & $21^{(2)}$ \\
$\McL$ & $\bf 22$ & $\bf 21$ & $21$ \\
$\McL{:}2$ & $\bf 22$ & $\bf 21^{(2)}$ & $21^{(2)}$ \\
$\He$ & $51^{(2)}$ \\ 
$\Ru$ & $\bf 28$ \\ 
$2.\Ru$ & & & $28^{(2)}$ \\ 
$2.\Suz$ & & $\bf 12$ \\ 
$2.\Suz.2^+$ & & $\bf 12^{(2)}$ \\ 
$2.\Suz.2^-$ & & $\bf 24$ \\ 
$3.\Suz$ & $\bf 24$ \\ 
$6.\Suz$ & & & $24$ & $\bf 12^{(2)}$ & & $\bf 12^{(2)}$ \\ 
$\Co_3$ & $\bf 22$ & $\bf 22$ & $23$ & $23$ \\ 
$\Co_2$ & $\bf 22$ & $\bf 23$ & $23$ & $23$ & $23$ \\ 
$\Co_1$ & $\bf 24$ \\
$2.\Co_1$ & & $\bf 24$ & $\bf 24$ & $\bf 24$ & $24$ & $24$ & $24$ \\ 
$\Fi_{22}$ & $78$ \\
$\Fi_{22}{:}2$ & $78$ \\
$3.\Fi_{22}$ & $54$ \\
\hline
\end{tabular}
\caption{$\mathbb{F}_pG$-modules with $\dim_{\mathbb{F}_p}(V)\leq u(G,p)$}
\label{tab:cases}
\end{table}

\begin{lemma}
\label{lemma:discard}
If $\dim_{\mathbb{F}_p}(V)$ is not bold in  Table $\ref{tab:cases}$, then $b(G)=1$.
\end{lemma}

\begin{proof}
\mbox{} 
\begin{enumerate}
\item[(i)]
If $(G,p,\dim_{\mathbb{F}_p}(V))$ is one of 
$$ (\M_{11},3,10),\,(2.\M_{22}.2^+,7,10),\,(2.\M_{22},7,10),\,(\HS{:}2,3,22),\,(\HS,3,22),
$$
then we use {\sc Magma} to prove that $b(G)=1$.

\item[(ii)]
If $(G,p,\dim_{\mathbb{F}_p}(V))$ is 
$(\M_{24},3, 22)$ or $(3.\Fi_{22},2,54)$, then 
we use {\sc Orb} directly to prove that $b(G)=1$. 

\item[(iii)]
If $(G,p,\dim_{\mathbb{F}_p}(V))=(\Co_2,5,23)$, then  we use {\sc Orb} with 
helper subgroup $2^4 \times 2^{1+6}.A_8$ 
to prove that $b(G)=1$. 
\end{enumerate}
Otherwise, we use {\sc Magma} to prove that inequality 
(\ref{eqn:strong}) holds with $n=1$, 
in which case $b(G)=1$ by Lemma~\ref{strong bound}.
\end{proof}

\begin{lemma}
\label{lemma:keep}
If $\dim_{\mathbb{F}_p}(V)$ is bold 
in Table $\ref{tab:cases}$, then one of the following holds.
\begin{enumerate}
\item $(G,p,\dim_{\mathbb{F}_p}(V),b(G))$ is  correctly 
listed in Table~$\ref{tab:totalex}$.
\item $(G,p,\dim_{\mathbb{F}_p}(V))=(2.\Co_1,3,24)$ and $b(G)\in \{2,3\}$.
\item $(G,p,\dim_{\mathbb{F}_p}(V))=(2.\Co_1,7,24)$ and $b(G)\leq 2$.
\end{enumerate}
\end{lemma}

\begin{proof}
Let $\ell:=\lceil (\log |G|)/\log (|V|-1)\rceil$, and recall that
$b(G)\geq\ell$ by Lemma \ref{lowerbound}. Let $n$ 
be the least positive integer such that (\ref{eqn:strong}) holds, 
so $b(G)\leq n$ by Lemma \ref{upperbound}.
For each relevant group we compute $n$ using {\sc Magma}.

\begin{enumerate}
\item[ (i)]
If $(G,p,\dim_{\mathbb{F}_p}(V))$ is one of  
$$ (\M_{11},3,5),\,(\M_{23},2,11),\,(\M_{24},2,11),\,(\J_2,2,12),\,
(\J_2{:}2,2,12), $$
$$ (3.\Suz,2,24),\,(\Co_3,2,22),\,(\Co_1,2,24), $$
then we use {\sc Magma} to prove that $b(G)=\ell$.

\item[(ii)]
If $(G,p,\dim_{\mathbb{F}_p}(V))$ is one of
$$ (\M_{11},3,10),\,(\M_{12},2,10),\,(\M_{12}{:}2,2,10),\,
(2.\M_{12},3,6),\,(2.\M_{12}.2^\pm,3,12),$$
$$
(\M_{22},2,10),\,(3.\M_{22},2,12),\,
(\J_1,2,20),\,(2.\J_2,3,14),\,(2.\J_2.2^-,3,14),$$
then $n>\ell$, and  we use {\sc Magma} to prove that $b(G)=n$.

\item[(iii)]
If $(G,p,\dim_{\mathbb{F}_p}(V))=(\M_{22}{:}2,2,10)$, then  $\ell=2$ and $n=4$.
We use {\sc Magma} to prove that $b(G)=3$. 

\item[(iv)] 
If $(G,p,\dim_{\mathbb{F}_p}(V))=(2.\Co_1,3,24)$, then $\ell=2$ and $n=3$, so (2) holds. Similarly, if $(G,p,\dim_{\mathbb{F}_p}(V))=(2.\Co_1,7,24)$, then $\ell=1$ and $n=2$, so (3) holds.

\item[(v)]
If $(G,p,\dim_{\mathbb{F}_p}(V))$ is one of  
$$ (\McL,3,21),\,(\McL{:}2,3,21),\,(6.\Suz,13,12), $$
then  $\ell=1$ and $n=2$. We use {\sc Orb}
and helper subgroups $\M_{11}, \M_{11}.2$, and 
$6.3^{(2+4)}.[48]$ 
respectively to prove that $b(G)=2$.
\end{enumerate}

In all other cases, $n=\ell$ so $b(G)=\ell$.
\end{proof}

\begin{lemma} \label{co1mod3}
If $(G,p,\dim_{\mathbb{F}_p}(V))=(2.\Co_1,3,24)$, then $b(G)=2$.	
\end{lemma}

\begin{proof}
By Lemma \ref{lemma:keep}, $b(G)\geq 2$.
Let $M\cong 2^{12}{:} \M_{24}$ be a maximal subgroup of $G$ 
and let $H\cong \M_{24}$ be a fixed
complement of $O_2(M)\cong 2^{12}$. As an $\mathbb{F}_2H$-module,
$O_2(M)$ is isomorphic to the binary Golay code, and so is uniserial
with descending composition series $11/1$. Thus all subgroups of $M$
properly containing $H$ necessarily contain $Z(M)=Z(G)$. 
Moreover, by \cite{Atlas},
every maximal subgroup $\tilde M$ of $G$ containing $H$
is $G$-conjugate to $M$, so we may assume that $\tilde M=M$.

We show that there exist $v,w\in V$ such that
$C_G(v)=H$ and $C_H(w)=\{1\}$.
The restriction of $V$ to $H$ is isomorphic to the natural 
permutation $\mathbb{F}_3H$-module, which is uniserial with descending 
composition series $1/22/1$. Hence there exists $0\neq v\in V$ 
fixed by $H$. Every subgroup of $G$ properly 
containing $H$ necessarily contains $Z(G)$, so it does not fix $v$.
This shows that $C_G(v)=H$. 
Lemma \ref{lemma:discard}
shows that the irreducible $H$-subquotient
of dimension $22$ of $V$ has a regular vector.
Hence $b(G) = 2$.
\end{proof}

\begin{lemma} \label{co1mod7}
If $(G,p,\dim_{\mathbb{F}_p}(V))=(2.\Co_1,7,24)$, then $b(G)=2$.	
\end{lemma}

\begin{proof}
By Lemma \ref{lemma:keep}, $b(G)\leq 2$.
To show that $G$ does not have a regular orbit on $V$, 
we use {\sc Orb} with 
the chain of helper subgroups
$\{1\}=U_0<U_1<U_2<U_3<U_4=G$ specified in Table \ref{tbl:help};
these are chosen so that $Z(G)<U_1$. 
The helper quotients of $V$, associated with the various helper subgroups, 
have dimension $d_i$. 

\begin{table}[!ht]\renewcommand{\baselinestretch}{1.1}\selectfont
\centering
$\begin{array}{rrrrr}
\hline 
i & U_i & |U_i| & [U_i {:} U_{i-1}] & d_i \\
\hline
4 & 2.\Co_1 & 8\,315\,553\,613\,086\,720\,000 & 46\,621\,575 & 24 \\
3 & (2\times 2_+^{1+8}).O_8^+(2) & 178\,362\,777\,600 & 270 & 16 \\
2 & (2^2.2^7).2^6.A_8 & 660\,602\,880 & 560 & 8 \\
1 & [2^{17}]{:} 3^2 & 1\,179\,648 & 1\,179\,648 & 8 \\
\hline
\end{array}$
\caption{Helper subgroups for $G$}\label{tbl:help}
\end{table}

Since the Cauchy-Frobenius Lemma shows that
there are $1097$ $G$-orbits on $V$, 
it is infeasible to enumerate sufficiently many (randomly chosen) 
orbits to rule out the existence of a regular orbit.
Instead, we consider the projective space ${\mathbb P}(V)$,
the set of $1$-spaces in $V$, which we view as a $G/Z(G)$-set.
Observe that $|{\mathbb P}(V)|=(7^{24}-1)/6\sim 3.2\cdot 10^{19}$;
to exclude a regular $G/Z(G)$-orbit on ${\mathbb P}(V)$ 
we must enumerate $1-|G|/(2\cdot|{\mathbb P}(V)|)\sim 87\%$ of it. 
With a random search we find representatives 
of $75$ $G$-orbits covering $\sim 90\%$ of ${\mathbb P}(V)$
without detecting a regular $G/Z(G)$-orbit.

But there might exist $v\in V$
having stabiliser $C_G(v)=\{1\}$ such that the $1$-space 
$\langle v\rangle\in{\mathbb P}(V)$ has non-trivial stabiliser
$C_G(\langle v\rangle)\cong{\mathbb F}_7^\ast$. 
Hence we must also exclude $G$-orbits of length $|G|/6$. To do this 
by an exhaustive enumeration of ${\mathbb P}(V)$, we must cover 
$1-|G|/(6\cdot|{\mathbb P}(V)|)\sim 96\%$ of the space.
The Cauchy-Frobenius Lemma shows that there are 
$382$ $G$-orbits on ${\mathbb P}(V)$, some of which may 
escape a random search.

Hence we proceed differently.
If $v$ is as above, then 
$C_G(\langle v\rangle)=\langle tz\rangle$, 
where $t\in G$ has order $3$, and $1\neq z\in Z(G)$. Now $v$ is an 
eigenvector for $t$ with respect to an eigenvalue $\omega$, where
$\omega\in\mathbb{F}_7^\ast$ is a primitive third root of unity.
By \cite{Atlas}, there are precisely four conjugacy classes of elements of
order $3$ in $G$; see Table \ref{tbl:co1cen}, where the associated
centralisers and the dimension $d_\omega(t)$ of the eigenspace
$E_\omega(t)$ of $t$ with respect to the eigenvalue $\omega$ are given.

\begin{table}[!ht]\renewcommand{\baselinestretch}{1.1}\selectfont
\centering
$\begin{array}{ccrr}
\hline 
t & d_\omega(t) & C_G(t) & |C_G(t)| \\
\hline
3A & 12 & 6.\text{Suz}               & 2\,690\,072\,985\,600 \\
3B &  6 & 2.(3^2.U_4(3).2)           &         117\,573\,120 \\
3C &  9 & 2.(3^{1+4}{:} 2.S_4(3)) &          25\,194\,240 \\ 
3D &  8 & 2.A_9 \times 3             &           1\,088\,640 \\
\hline
\end{array}$
\caption{$3$-centralisers in $G$}\label{tbl:co1cen}
\end{table}

Since the conjugacy classes under consideration are rational, 
the normaliser $N_G(\langle t\rangle)$ interchanges the eigenspaces 
$E_\omega(t)$ and $E_{\omega^2}(t)$, so it suffices to consider
either of the primitive third roots of unity in $\mathbb{F}_7^\ast$.
Now $C_G(t)$ acts on $E_\omega(t)$, and hence 
$C_G(t)/\langle tz\rangle$ acts on ${\mathbb P}(E_\omega(t))$. 
Since $C_{C_G(t)}(\langle v\rangle)=C_G(\langle v\rangle)=\langle tz\rangle$, 
we conclude that $\langle v\rangle$ belongs to a regular 
$C_G(t)/\langle tz\rangle$-orbit on ${\mathbb P}(E_\omega(t))$.

For conjugacy classes $3A$ and $3B$, we check that
$|{\mathbb P}(E_\omega(t))|=(7^{d_\omega(t)}-1)/6<|C_G(t)|/6$,
hence there cannot be a regular $C_G(t)/\langle tz\rangle$-orbit.
To deal with classes $3C$ and $3D$, 
we pick a representative $t$, compute the action of $C_G(t)$ on 
$E_\omega(t)$, and enumerate ${\mathbb P}(E_\omega(t))$ 
completely by a standard orbit computation.  
For conjugacy class $3C$, none of the 
$21$ $C_G(t)/\langle tz\rangle$-orbits in ${\mathbb P}(E_\omega(t))$ 
is regular; for conjugacy class $3D$ precisely one of the $26$ 
relevant orbits is regular.
We pick $\langle v\rangle$ 
from this unique regular $C_G(t)/\langle tz\rangle$-orbit on 
${\mathbb P}(E_\omega(t))$, and use {\sc Orb} 
to enumerate $51\%$ of the $G$-orbit of $\langle v\rangle$ in 
${\mathbb P}(V)$. This shows that
$C_G(\langle v\rangle)\cong Z(G)\times A_4$. 
Hence $b(G) = 2$.
\end{proof} 


\end{document}